\documentclass[12pt,reqno]{amsart}
\usepackage{amsmath, amsthm, amscd, amsfonts, amssymb, 
graphicx, color}
\usepackage[bookmarksnumbered, colorlinks, plainpages]{hyperref}
\usepackage{plain}
 
\def\rn{\mathbb{R}^n}
\def\rd{\mathbb{R}^d}

\renewcommand{\epsilon}{\varepsilon}

\newcommand{\uqh}{\overline{\dim}_{qH}} 
\newcommand{\lqh}{\underline{\dim}_{qH}} 

\newcommand\pot{^}
\newcommand\ubd{\overline{\mbox{\rm dim}}_{\rm B}\,}  
\newcommand\lbd{\underline{\mbox{\rm dim}}_{\rm B}\,}  
\newcommand\da{\underline{\mbox{\rm dim}}_{\,\theta}}

\newcommand\uda{\overline{\mbox{\rm dim}}_{\,\theta}}  
\newcommand\dap{\underline{\dim}_{\theta}^{m}}
\newcommand\udap{\overline{\dim}_{\theta}^{m}}

\newcommand{\dimh}{\dim_{\rm H}}

\newcommand{\be}{\begin{equation}}  
	\newcommand{\ee}{\end{equation}}

\newcommand{\indicator}{\mathbb{I}}
\DeclareFontFamily{U}{stix2bb}{}
\DeclareFontShape{U}{stix2bb}{m}{n} {<-> stix2-mathbb}{}

\newtheorem{theo}{Theorem}[section]
\newtheorem{cor}[theo]{Corollary}
\newtheorem{lem}[theo]{Lemma}

\newtheorem{defi}[theo]{Definition}

\newtheorem{exa}[theo]{Example}

\newtheorem{rmk}[theo]{Remark}
\DeclareFontFamily{U}{stix2bb}{}
\DeclareFontShape{U}{stix2bb}{m}{n} {<-> stix2-mathbb}{}

\begin{document}

\title [Critical values for intermediate dimensions of projections]{ Critical values for intermediate and box dimensions of projections and other images of compact sets}
\date{\today}
\author{Nicolas E. Angelini} \address{Universidad Nacional de San Luis \\Facultad de Ciencias F\'\i sico Matem\'aticas y Naturales\\Departamento de Matem\'atica\\ and \\CONICET \\Instituto de Matem\'atica Aplicada,\\ San Luis (IMASL)}
\email{nicolas.angelini.2015@gmail.com}
\author{Ursula M. Molter}
\address{Universidad de Buenos Aires,\\ Facultad de Ciencias Exactas y Naturales,\\ Departamento de Matem\'atica, \\and\\ UBA-CONICET,\\ Instituto de Investigaciones Matem\'aticas,\\ Luis A. Santal\'o (IMAS)}
\email{umolter@conicet.gov.ar}

\begin{abstract}
    Given a compact set $E \subset \rd$, we investigate for which values of $m$ the equality $\dim_\theta P_V(E) = m$ or $\dim_\theta P_V(E) = \dim_\theta E$ holds for $\gamma_{d,m}$-almost all $V \in G(d,m)$. Our result extends to more general functions, including orthogonal projections and fractional Brownian motion. As a particular case, when $\theta = 1$, the results apply to the box dimension.
\end{abstract}
\maketitle

\section{Introduction}

The variation of the dimension of a set with respect to its orthogonal projection onto $m$-dimensional linear subspaces of $\rd$ is a classical problem in geometric measure theory, and there is a substantial body of results on this topic. In short, if $G(d,m)$ denotes the set of all $m$-dimensional subspaces of $\rd$, the question is: How is $\dim P_V(E)$ related to $\dim E$, where $E \subset \rd$, $V \in G(d,m)$, and $P_V$ denotes the orthogonal projection onto $V$.

When $\dim$ represents the Hausdorff dimension, the problem has already been studied, yielding the result that if $E \subset \rd$ is a Borel set and $P_V$ denotes the orthogonal projection onto an $m$-dimensional linear subspace $V$ with $m \leq d$, then
\be\label{Critical value for Hausdorff dimension}
    \dimh P_V E = \min \{ \dimh E, m \}
\ee
for $\gamma_{d,m}$-almost all $V \in G(d,m)$. This result was first obtained for the case $d = 2$, $m = 1$ by Marstrand \cite{Marstrand1954SomeFG} and later extended to Borel sets $E \subset \rd$, $d \geq 2$ by Mattila \cite{Mattila1975HausdorffDO}.

The lower and upper $\theta$-intermediate dimensions, $\da$ and $\uda$ respectively, form a continuous family of dimensions that interpolate between the Hausdorff and the box dimensions of bounded sets in $\rd$. They satisfy, for all $\theta \in (0,1]$, the inequalities $\dim_H E \leq \da E \leq \uda E \leq \ubd E$ and $\da E \leq \lbd E$, with $\overline{\dim}_1 E = \ubd E$. For $\theta = 1$, we have $\underline{\dim}_1 E = \lbd E$. These dimensions are continuous functions of $\theta$ for all $\theta \in (0,1]$; however, in general, the $\theta$-intermediate dimension does not converge to the Hausdorff dimension as $\theta \to 0$.

For box and $\theta$-intermediate dimensions, general results such as \eqref{Critical value for Hausdorff dimension} do not hold. Nevertheless, as shown in \cite{FalconerProj} for the box dimension and in \cite{Burrell_2021} for $\theta$-intermediate dimensions, for any fixed $\theta \in (0,1]$, the value of $\dim_\theta P_V(E)$ remains invariant for $\gamma_{d,m}$-almost every \( V \in G(d,m) \). Recall that $\dim_1 P_V(E) = \dim_B P_V(E)$.

These values of the intermediate dimension are referred to by the authors as the ``\textit{(Upper or Lower) $\theta$-intermediate dimension profiles},'' and are denoted by $\overline{\dim}^{m}_\theta$ and $\underline{\dim}^{m}_\theta$, respectively. In general, these profiles are difficult to work with.

When dealing with projections, the superscript $m$ in the dimension profiles typically denotes an integer. However, in other contexts -- as we will see -- the superscript may represent a real number. To distinguish between these cases, from now on we will use $t$ for real numbers and $m$ for integers when referring to dimension profiles. 

The precise result obtained in \cite{FalconerProj} and \cite{Burrell_2021} is:

\begin{theo}
    Let $E \subset \rd$ be bounded. Then for all $V \in G(d,m)$,
    \begin{center}
        $\da P_V(E) \leq \underline{\dim}^{m}_\theta E$ and $\uda P_V(E) \leq \overline{\dim}^{m}_\theta E$
    \end{center}
    for all $\theta \in (0,1]$, and for $\gamma_{d,m}$-almost all $V \in G(d,m)$,
    \begin{center}
        $\da P_V(E) = \underline{\dim}^{m}_\theta E$ and $\uda P_V(E) = \overline{\dim}^{m}_\theta E$.
    \end{center}
\end{theo}

The reader can also refer to \cite{feng2024intermediate} for additional results on $\theta$-intermediate dimensions of projections.

In the present paper, 
we study for which values of $m$, or under which conditions on the set $E$, results analogous to Marstrand’s theorem can be obtained for intermediate and box dimensions.

We introduce the notions of the \emph{upper} and \emph{lower quasi-Hausdorff dimensions} of a set $E$, denoted by $\uqh E$ and $\lqh E$, respectively, as 
$$\uqh E := \lim_{\theta \to 0} \overline{\dim}_\theta E \quad \text{and}\quad \lqh E := \lim_{\theta \to 0} \underline{\dim}_\theta E.$$ 
This helps us make the statements and proofs of our results clearer and more straightforward. 

We show that, given a compact set \( E \subset \mathbb{R}^d \) and a positive integer \( m \leq d \), the almost sure value of the upper and lower intermediate dimensions of the orthogonal projection onto an \( m \)-dimensional subspace satisfies
\[
\uqh P_V E = \min\left\{m, \uqh E \right\}
\quad \text{for } \gamma_{d,m}\text{-almost every } V \in G(d,m),
\]
and similarly,
\[
\lqh P_V E = \min\left\{m, \lqh E \right\}
\quad \text{for } \gamma_{d,m}\text{-almost every } V \in G(d,m),
\]
see Theorem \ref{main theorem}. From this, we deduce:
\[
\ubd P_V(E) = m \text{ for } \gamma_{d,m}\text{-almost all } V \in G(d,m) \iff m \leq \uqh E,
\]
and
\[
\lbd P_V(E) = m \text{ for } \gamma_{d,m}\text{-almost all } V \in G(d,m) \iff m \leq \lqh E.
\]

A direct consequence of Theorem \ref{main theorem} is that if \( E \) satisfies \( \lbd E = \lqh E \), then
\[
\lbd(P_V E) = \min\{m, \lbd E\}
\]
for almost every \( V \in G(d,m) \).

Note that in \cite[Corollary 1.3]{falconer2021assouad} the authors show that if \( \underline{\dim}_\mathrm{B} E = \dim_{qA} E \), then
\[
\underline{\dim}_\mathrm{B}(P_V E) = \min\{\underline{\dim}_\mathrm{B} E, m\}
\]
for almost every \( V \in G(d,m) \).

Since by Theorem \ref{Angelini ineq} we have that, $\underline{\dim}_\mathrm{B} E = \dim_{qA} E \quad \text{ implies } \quad \underline{\dim}_\mathrm{B} E = \lqh E,$
and the converse is not necessarily true, the result proven in this paper is strictly more general than the one in \cite{falconer2021assouad}.
All the above statements remain valid if one replaces the lower box dimension and the lower quasi-Hausdorff dimension with their upper counterparts.

We also investigate for which values of $m$ the intermediate dimension and the intermediate dimension profiles coincide. We obtain a general lower bound for the profiles in terms of the Assouad and Assouad dimension spectrum $\dim^{\alpha}_A$:
\[
\overline{\dim}_{\theta}^{t} E \geq \uda E - \max\{0, \dim_A^{\alpha} E - t, (\dim_A E - t)(1 - \alpha)\},
\]
for all $\alpha \in (0,1)$, which yields that if $m \geq \dim_{qA} E$, then for all $\theta \in (0,1]$ we have $\dim_\theta E = \dim_\theta P_V E$. This result was previously known only in the case of the box dimension.

Another known bound that we adapt to intermediate dimensions is the inequality:
\[
\overline{\dim}_{\theta}^{s} E \geq \frac{\overline{\dim}_{\theta}^{t} E}{1 + \left(\frac{1}{s} - \frac{1}{t}\right)\overline{\dim}_{\theta}^{t} E},
\]
and we demonstrate with a simple example that this inequality is sharp, with equality holding for certain sets.

Since these profiles are also related to the dimensions of more general families of functions -- such as fractional Brownian motion, as shown in \cite{Brownian1} -- we can extend all our results to such functions.

Theorem \ref{main theorem} and its corollaries provide genuinely new insights only in cases where the intermediate dimensions of the set are not continuous at $\theta = 0$. To illustrate this point, we now present an example of a set for which our projection theorems yield information that is not captured by existing results.

\begin{exa}
Let \( n \geq 2 \). For each \( 1 \leq j < n \), it is not difficult to construct a set \( E \subset \mathbb{R}^n \) such that
\[
\dim_H E < j \leq \lqh E.
\]
Indeed, let \( F_1 \subset \mathbb{R}^j \) be a set with \( j - 1 < \dim_H F_1 = \dim_B F_1 \leq j \); for instance, \( F_1 \) could be a self-similar set. Define
\[
F_2 = \left(\{0\} \cup \left\{ \frac{1}{\log(n)} \right\}_{n \in \mathbb{N}} \right)^{n-j} \subset \mathbb{R}^{n-j}.
\]
Then \( F_2 \) is a countable set such that \( \da F_2 = \overline{\dim}_\theta F_2 = n - j \) for all \( \theta \in (0,1] \). Now define \( E = F_1 \times F_2 \). Then
\[
\dim_H E = \dim_H F_1 \leq j < n - j + \dim_H F_1 = \lqh E,
\]
where the last equality follows from the product rule for intermediate dimensions (see \cite[Theorem 5.4]{Banaji}).

Therefore, by Theorem~\ref{main theorem}, we conclude that
\[
\lqh P_V E = \uqh P_V E= j \quad \text{for } \gamma_{n,j}\text{-almost every } V \in G(n,j).
\]
In contrast, Marstrand's projection theorem gives 
\[
\dim_H P_V E = \dim_H E < j \quad \text{for } \gamma_{n,j}\text{-almost every } V \in G(n,j).
\]
\end{exa}

\section{Preliminaries}\label{preliminares}

Throughout the whole document $B(x,r)$, $r>0$ will denote the open ball in $\rd$ with center $x$ and radius $r$. Given a non-empty set $E\subset\rd$, $\indicator_E (x)$ will represent the indicator function of $E$, i.e. $\indicator_E (x)=1$ if $x\in E$ and $\indicator_E(x)=0$ if $x\not\in E$.

We will write $|E|$ for the diameter of the set, $\dimh E$ will always represent the Hausdorff dimension while $\ubd E$ and $\lbd E$ will represent the upper and lower box dimension respectively, see \cite{FBook} for more information on these dimension.

The support of a measure $\mu$ on $\rd$ is defined as as usual,
$$spt\,\mu :=\rd\setminus\{x:\exists\, r>0\mbox{ such that } \mu(B(x,r))=0\}.$$

$G(d,m)$ will denote the manifold of all $m$ dimensional linear sub-spaces of $\rd$ and $\gamma_{d,m}$ the Haar measure on $G(d,m)$, see section 3 on \cite{M95} for more information.
Given a compact set $E\subset \rd$ and $V\in G(d,m)$, $P_V(E)$ will be the orthogonal projection of $E$ onto $V$.

The $\theta$ intermediate dimensions were introduced in \cite{FFK} and are defined as follows.

\begin{defi}\label{adef}
Let $F\subseteq \mathbb{R}\pot{d}$ be bounded. For $0\leq \theta \leq 1$ we define the {\em  lower $\theta$-intermediate dimension} of $F$ by
\begin{align*}
\da F =  \inf \big\{& s\geq 0  : \ \forall\  \epsilon >0, \ {\rm and\ all }\  \delta_0>0, {\rm there}\ \exists\ 0<\delta\leq \delta_0, \ { \rm and }\  \\
 & \{U_i\}_{i\in I}\  : F \subseteq \cup_{i\in I}U_i \ : \delta^{1/\theta} \leq  |U_i| \leq \delta \ {\rm and }\ \sum_{i\in I} |U_i|^s \leq \epsilon  \big\}.
\end{align*}

Similarly, we define the {\em  upper $\theta$-intermediate dimension} of $F$ by
\begin{align*}
\uda F =  \inf \big\{& s\geq 0  : \ \forall\  \epsilon >0, \ {\rm there }\ \exists\ \delta_0>0, \ : \ \forall\ 0<\delta\leq \delta_0, \mbox{ \rm there }\  \exists \\
 & \{U_i\}_{i\in I}\  : F \subseteq \cup_{i\in I}U_i \ : \delta^{1/\theta} \leq  |U_i| \leq \delta \ {\rm and }\ \sum_{i\in I} |U_i|^s \leq \epsilon  \big\}.
\end{align*}
\end{defi}

For all $\theta\in (0,1]$ and a compact set $A\subset\rd$ we have $\dimh A\leq \uda A\leq \ubd A$ and similarly with the Lower case replacing $\ubd$ with $\lbd$. This spectrum of dimensions has the property of being continuous for $\theta\in(0,1]$ with $\underline{\dim}_1 E = \lbd E$ and $\overline{\dim}_1 E =\ubd E$, leaving the natural question of the continuity in $\theta=0$, i.e. $\displaystyle\lim_{\theta\to 0}\da E=\dimh E$ or $\displaystyle\lim_{\theta\to 0}\uda E=\dimh E$, as a problem to study. In general these equalities are not true. The reader can find more information in \cite{FFK}.

As mentioned in the introduction, we will adopt a specific notation for the limiting case of the intermediate dimension as the parameter $\theta$ tends to $0$. We define the following quantities, which we refer to as the \textit{quasi-Hausdorff dimensions}:

\begin{defi}
Let \( E \subset \mathbb{R}^d \) be a non-empty bounded set. We define the \emph{lower quasi-Hausdorff dimension} and the \emph{upper quasi-Hausdorff dimension} of \( E \) as
\[
\lqh E := \lim_{\theta \to 0} \da E
\quad \text{and} \quad
\uqh E := \lim_{\theta \to 0} \uda E,
\]

If both limits coincide, we say that the \emph{quasi-Hausdorff dimension} of \( E \) exists, and we denote it by
\[
\dim_{qH} E := \lqh E = \uqh E.
\]
\end{defi}

There is also an equivalent definition of intermediate dimension. Working with it will be more beneficial to our objectives. For $E\subset \rd$ bounded and non-empty, $\theta\in (0,1]$, $r>0$ and $s\in[0,d]$, define

$$ S^{s}_{r,\theta}(E)=\inf\left\{\displaystyle\sum_i |U_i|^s : \{U_i\}_i \text{ is a cover of }E\text{ such that } r\leq |U_i|\leq r^\theta \text{ for all } i\right\}.$$

Then we have 

$$\da E=\left( \text{ the unique } s\in[0,d] \text{ such that } \displaystyle\liminf_{r\to 0} \frac{\log S^{s}_{r,\theta}(E)}{-\log(r)}=0\right)$$
and 
$$\uda E=\left( \text{ the unique } s\in[0,d] \text{ such that } \displaystyle\limsup_{r\to 0} \frac{\log S^{s}_{r,\theta}(E)}{-\log(r)}=0\right).$$

Furthermore, the next result holds

\begin{lem}\cite[Lemma 2.1]{Burrell_2021} Let $\theta\in (0,1]$ and $E\subset\rd$. For each $0<r<1$,
$$-(s-t)\leq \frac{\log S^{s}_{r,\theta}(E)}{-\log(r)} - \frac{\log S^{t}_{r,\theta}(E)}{-\log(r)}\leq -\theta (s-t) \quad (0\leq t\leq s\leq d).$$

\begin{rmk}\label{s leq uda}
    Note that this implies that if $\displaystyle\liminf_{r\to 0} \frac{\log S^{s}_{r,\theta}(E)}{-\log(r)}\geq 0$ then  $s\leq \da E$, and analogously, if $\displaystyle\limsup_{r\to 0} \frac{\log S^{s}_{r,\theta}(E)}{-\log(r)}\geq 0$ then  $s\leq \uda E$.
\end{rmk}

\end{lem}

Let $E\subset\rd$ be a non-empty set and let $N_r(E)$ be the minimum number of sets of diameter $r$ that can cover
$E$. The \textit{Assouad dimension} of $E$ is defined by  \begin{align*}
    \dim_A E =\inf \big\{ s : \exists \ C>0\ \text{such that } \ &\forall \ 
     0<r<R\mbox{ and }\\
     & \hspace{-1cm} x\in E, N_r(B(x,R)\cap E)\leq C\left({R}/{r}\right)\pot{s} \big\} 
\end{align*}
and the \textit{upper Assouad spectrum} of $E$, for $\alpha \in (0,1)$ is defined by  
\begin{align*}
    \overline\dim_A\pot{\alpha} E =\inf \big\{ t :\exists\ C>0 \ \text{such that } 
    & \forall\  0<r\leq R\pot{1/\alpha} <R<1\mbox{ and } \\
    & \hspace{-1cm} x\in E, N_r(B(x,R)\cap E)\leq C\left({R}/{r}\right)\pot{t} \big\}. 
\end{align*}
The upper Assouad spectrum is non decreasing in $\alpha$. Finally, the \textit{quasi-Assouad dimension}, introduced in \cite{lu2016quasi}, is defined by $$\dim_{qA}E=\displaystyle\lim_{\alpha\to 1} \overline{\dim}_A\pot{\alpha}E$$ and we have for $\alpha\in (0,1)$
$$\ubd E\leq \overline{\dim}_A\pot{\alpha}E  \leq \dim_{qA}E\leq  \dim_A E.$$
For background on Assouad-type dimensions, the reader can refer to \cite{fraser2020assouad}.

Given a compact set $E\subset\rd$ and $1\leq m\leq d$, in \cite{Burrell_2021} the lower and upper $\theta$- intermediate dimension profile, $\dap E$ and $\udap E$ respectively, were introduced in order to denote the value of the $\theta$ intermediate dimension (lower or upper respectively) of $P_V (E)$, for $\gamma_{d,m}$ a.e. $V\in G(d,m)$.

In fact, the next results holds.

\begin{theo}\cite[Theorem 5.1]{Burrell_2021} \label{Projection theorem for intermediate dimension}
    Let $E\subset\rd$ be bounded. Then, for all $V\in G(d,m)$
    $$\underline{\dim}_\theta P_V E\leq \underline{\dim}_{\theta}^{m} E \hspace{0.2 cm}\mbox{   and   } \hspace{0.2 cm}\overline{\dim}_\theta P_V E\leq \overline{\dim}_{\theta}^{m} E,$$
    for all $\theta\in (0,1].$ Moreover, for $\gamma_{d,m}-$ almost all $V\in G(d,m)$,
      $$\underline{\dim}_\theta P_V E= \underline{\dim}_{\theta}^{m} E\hspace{0.2 cm}\mbox{  and  } \hspace{0.2 cm}\overline{\dim}_\theta P_V E= \overline{\dim}_{\theta}^{m}E,$$
      for all $\theta\in(0,1].$
\end{theo}

\begin{theo}\cite[Theorem 3.9]{Brownian1} \label{exceptional directions}
Let $E\subset\rd$ be compact, $m\in\{1,...,d\}$ and $0\leq\lambda\leq m$, then
$$\dimh\{V\in G(d,m):\uda P_V E<\overline{\dim}_{\theta}^{\lambda}E\}\leq m(d-m)-(m-\lambda)$$
and
$$\dimh\{V\in G(d,m):\da P_V E<\underline{\dim}_{\theta}^{\lambda}E\}\leq m(d-m)-(m-\lambda).$$
    
\end{theo}

Results on the exceptional directions for the box dimension of projections of sets were first obtained in \cite{falconer2021capacity}.
Later, in \cite{Brownian1}, the $\theta$ intermediate profile was generalized from considering integers $m$ to consider positive real numbers $t>0$.  These new possible values turn out to appear for the case of the almost sure $\theta$ intermediate dimension of the image of the set $E$ under the fractional Brownian motion.

We recall the definition of index$-\alpha$ fractional Brownian motion, which, following the notation of \cite{Brownian1}, we denote by $B_\alpha:\rd\to\mathbb{R}^m$ for $d,m\in\mathbb{N}$. $B_\alpha=(B_{\alpha,1},...,B_{\alpha,m})$, where $B_{\alpha,i}:\mathbb{R}^d\to \mathbb{R}$ for each $i$. They satisfy:
\begin{itemize}
    \item $B_{\alpha,i}(0)=0$;
    \item $B_{\alpha,i}$ is continuous with probability $1$;
    \item the increments $B_{\alpha,i}(x)-B_{\alpha,i}(y)$ are normally distributed with mean $0$ and variance $|x-y|^{2\alpha}$ for all $x,y \in\rd$.
\end{itemize}
Moreover, $B_{\alpha,i}$ and $B_{\alpha,j}$ are independent for all $i,j\in \{1,...,m\}$.\\
The next result holds.

\begin{theo}\cite[Theorem 3.4]{Brownian1} \label{dimension of Fractional Brownian motion}
    Let $\theta\in (0,1]$, $m,d\in\mathbb{N}$, $B_\alpha:\rd\to\mathbb{R}^{m}$ be index-$\alpha$ fractional Brownian motion $(0<\alpha<1)$ and $E\subset\rd$ be compact. Then almost surely
    $$\da B_{\alpha}(E)=\frac{1}{\alpha}\underline{\dim}_{\theta}^{m\alpha}E 
    \quad \text{ and }\quad 
     \uda B_{\alpha}(E)=\frac{1}{\alpha}\overline{\dim}_{\theta}^{m\alpha}E.$$
\end{theo}

To define the intermediate dimension profiles we first need to define the necessary kernel functions.
\begin{defi}
    Let $\theta \in (0,1],t>0, 0\leq s\leq t$ and $0<r<1$, define $\phi_{r,\theta}\pot{s,t}(x)$ by 
    \begin{equation*}
\phi_{r,\theta}\pot{s,t}(x)=
    \begin{cases}
        1 & \text{if } 0\leq |x|<r\\
        \left(\frac{r}{|x|}\right)\pot{s} & \text{if } r\leq |x|\leq r\pot{\theta}\\
        \frac{r\pot{\theta(t-s)+s}}{|x|\pot{t}} & \text{if } r\pot{\theta}\leq |x|
    \end{cases}
\end{equation*}
and $$C_{r,\theta}\pot{s,t}(E)=\left( \displaystyle\inf_{\mu\in\mathcal{M}(E)}\iint \phi_{r,\theta}\pot{s,t}(x-y) \, d\mu (x) \, d\mu (y)\right)\pot{-1} $$
where $\mathcal{M}(E)$ denotes the set of probability measures supported on $E$.
\end{defi}

The lower and upper $\theta$ intermediate dimension profiles of a bounded set $E\subset\rd$ are defined as follows.
\begin{defi}\label{intermediate idmension profile}
    Let $t>0$ and $E\subset\rd$. The \textit{lower intermediate dimension profile} of $E$ is defined as
    $$\underline\dim_\theta\pot{t}E=\left( \mbox{ The unique } s\in[0,t]\mbox{ such that }  \liminf_{r\to 0}\frac{\log C_{r,\theta}\pot{s,t}(E)}{-\log r} = s\right) $$
    and the \textit{upper intermediate dimension profile} of $E$ is $$\overline\dim_\theta\pot{t}E=\left( \mbox{ The unique } s\in[0,t]\mbox{ such that }  \limsup_{r\to 0}\frac{\log C_{r,\theta}\pot{s,t}(E)}{-\log r} = s\right). $$
\end{defi}

\begin{rmk}\label{cota inferior}
    By \cite[Lemma 3.2]{Burrell_2021} and \cite[Lemma 2.2]{Brownian1} we have that 
    $$\text{if}\  0\leq \liminf_{r\to 0 } \frac{\log C_{r,\theta}\pot{s,t}(E)}{-\log r} -s, \quad \text{then} \quad s\leq \underline\dim_\theta\pot{t}E,$$ and  
    $$\text{if}\  0\leq \limsup_{r\to 0 } \frac{\log C_{r,\theta}\pot{s,t}(E)}{-\log r} -s, \quad \text{then} \quad s\leq \overline\dim_\theta\pot{t}E.$$ 
\end{rmk}
Finally we give a result that we will need later.

\begin{cor} \cite[Corollary 3.14]{Banaji}\label{banajis bound}
   Let $E\subset\rd$ bounded. Then $$\uda E\geq \frac{\theta\, d\,\ubd E}{d-(1-\theta)\ubd E}.$$ 
   The same holds replacing $\uda$ with $\da $ and $\ubd$ with $\lbd$ respectively.
\end{cor}

\begin{lem}\cite[Proposition 2.3]{FFK} and \cite[Lemma 5.2]{Banaji}\label{Frostman lemma}
    Let $E$ be a compact subset of $\rd$, let $0<\theta\leq 1$, and suppose $0<s<\da E$. There exists a constant $c>0$ such that for all $r\in(0,1)$ we can find a Borel probability measure $\mu_r$ supported on $E$ such that for all $x\in\rd$ and $r\leq \delta \leq r\pot{\theta}  $,
    \be\label{frostman}\mu_{r}(B(x,\delta))\leq c\, \delta\pot{s}. \ee
    Analogously, if $0<s<\uda E$, we have that there exists a constant $c>0$ such that for all $r_0>0$ there exists $r\in(0,r_0)$ and a Borel probability measure $\mu_r$ supported on $E$ satisfying  \eqref{frostman} for $r\leq \delta\leq r^{\theta}$ and all $x\in\rd$.
\end{lem}

We finish the preliminaries by recalling another family  of kernels $\tilde{\phi}^{s}_{r,\theta}$ on $\mathbb{R}^d$, defined for $0<r<1$, $\theta\in[0,1]$ and $0<s\leq m$ in \cite{Burrell_2021} and \cite{Brownian2}
$$\tilde{\phi}^{s}_{r,\theta}= \left\{ \begin{array}{lcc} 1 & if & |x|<r  \\ \\ \left(\frac{r}{|x|}\right)^s & if & r\leq |x|<r^\theta \\ \\ 0 & if & r^\theta \leq |x|.\end{array} \right.$$

These kernels are very important because they are useful in order to bound the sums of the diameters of coverings. The next result holds.
\begin{lem}\cite[Lemma 2.1]{Brownian2} \label{lemma intermediate dimension}
    Let $E\subset \rd$ be compact, $\theta \in (0,1]$, $0<r<1$ and $0\leq s\leq d$ and let $\mu\in\mathcal{M}(E)$. Then 
    $$S^{s}_{r,\theta}(E)\geq r^s \left[\iint \tilde{\phi}^{s}_{r,\theta}(x-y)d\mu(x)d\mu(y) \right]^{-1}.$$
\end{lem}

\section{Results}
Our first result is a version of \textit{Marstrand's theorem} for quasi-Hausdorff dimension. If $\dim_\theta E$ is continuous at $\theta=0$, then our result recovers the classical Marstrand-–Mattila theorem. In general, however, we have
\[
\dim_H E \leq \lqh E \leq \uqh E,
\]
and these inequalities can be strict. 

\begin{theo}\label{main theorem}
Let $E\subset\rd$ be a compact set and $0\leq m\leq d$. Then
\[
\min\left\{ m,\lqh E\right\} = \lqh P_V E \quad \text{for } \gamma_{d,m}\text{-almost all } V\in G(d,m),
\]
and 
\[
\min\left\{ m,\uqh E\right\} = \uqh P_V E \quad \text{for } \gamma_{d,m}\text{-almost all } V\in G(d,m).
\]
\end{theo}

\begin{proof}
We will prove the case of $\da E$; the argument for $\uda E$ is similar.  
We formulate the proof in terms of intermediate dimension profiles. That is, we show that  
\[
\min\left\{ m, \displaystyle\lim_{\theta \to 0} \da E \right\} = \lim_{\theta\to 0} \underline{\dim}_{\theta}^{m} E.
\]
Recall that the intermediate dimension profile gives the almost sure value of the intermediate dimension of orthogonal projections.

Since for all $0\leq m\leq d$, we have $\underline{\dim}_{\theta}^{m}E \leq m$ and $\underline{\dim}_{\theta}^{m}E \leq \da E$, it follows that $\displaystyle\lim_{\theta\to 0}\underline{\dim}^{m}_{\theta} E$ is always less than or equal to the left-hand side.

We now prove the reverse inequality. Assume, without loss of generality, that $|E|<1$.

Let $d(E):=\displaystyle\lim_{\theta\to 0} \da E$, and choose $t<\min\left\{ m,d(E)\right\}$ and $\epsilon>0$ small enough so that $t+\epsilon<\min\left\{ m,d(E)\right\}$.

Then there exists $\theta_0 >0$ such that $t+\epsilon < \da E$ for all $\theta\in (0,\theta_0)$.

Let $r>0$, $\theta\in (0,1)$, and choose $\theta_1 \in \left(0, \min\{\theta,\theta_0\} \right)$ such that $|E|<r^{\theta_1}$.

Then, by Lemma \ref{Frostman lemma}, there exists a constant $c>0$ and a probability measure $\mu$ supported on $E$ such that 
\be\label{frostman ineq} \mu(B(x,u))\leq c\, u^{t+\epsilon}\ee
for all $x\in \rd$ and $u\in[r,r^{\theta_1}]$.

Now, integrating the kernel $\phi_{r,\theta}^{t,m}$ with respect to $\mu$, and using $\indicator_{A}$ for the characteristic function of the set $A$, we have
\begin{align*}
    \left(C_{r,\theta}\pot{t,m}(E)\right)\pot{-1}&\leq \iint \phi_{r,\theta}^{t,m}(x-y)d\mu(y) d\mu (x)\\
    &=\int \mu(B(x,r))d\mu(x)+r^{t}\iint \frac{\indicator_{(B(x,r^{\theta})\setminus B(x,r))}(y)}{|x-y|^{t}} d\mu(y)d\mu(x) \\
    & \qquad
    + r^{\theta(m-t)+t} \iint \frac{\indicator_{(\rd\setminus B(x,r^{\theta}))}(y)}{|x-y|^{m}}d\mu(y) d\mu(x)\\
   &=S_1 + S_2 +S_3.
\end{align*}

By \eqref{frostman ineq}, we have 
$$S_1\leq c\, r^{t+\epsilon}<c\, r^{t}.$$

For $S_2$, using a change of variables and again \eqref{frostman ineq}, we obtain
\begin{align*}
     \int &\frac{\indicator_{(B(x,r^{\theta})\setminus B(x,r))}(y)}{|x-y|^{t}} d\mu(y)
     = \int  \mu\left\lbrace y:\frac{\indicator_{(B(x,r^{\theta})\setminus B(x,r))}(y)}{|x-y|^{t}}\geq u\right\rbrace du\\
     &=\int_0^{1/r^{\theta t}} \mu(B(x,r^{\theta})\setminus B(x,r)) du+ \int_{1/r^{\theta t}}^{1/r^{t}} \mu\left\lbrace y:\frac{1}{|x-y|^{t}}\geq u\right\rbrace du \\
     &\leq c\,r^{\theta \epsilon}+ t \int_{r}^{r^{\theta}}s^{-t-1} \mu(B(x,s)) ds
    \leq c\, r^{\theta \epsilon}+ \frac{c\,t}{\epsilon}\left( r^{\theta\epsilon }-r^{\epsilon} \right)\\
     &< c+ \frac{c\,t}{\epsilon},
\end{align*}

and hence 
$$S_2< \left( c+ \frac{c\, t}{\epsilon}\right) r^{t}.$$

For $S_3$, since $|E|<r^{\theta_1}$ and $\mu$ is supported on $E$, we have 
\begin{align*}
    & \int \frac{\indicator_{(\rd\setminus B(x,r^{\theta}))}(y)}{|x-y|^{m}} d\mu(y)\\
    &=\int_{0}^{1/r^{\theta m}} \mu (B(x,u^{-1/m})\setminus B(x,r^{\theta}))\, du\\
    &\leq \int_{0}^{1/|E|^{ m}} 1\, du + \int_{1/|E|^{ m}}^{1/r^{\theta m}} \mu (B(x,u^{-1/m}))\, du\\ 
    &= |E|^{-m} + m\int_{r^{\theta}}^{|E|} s^{-m-1}\mu (B(x,s))\, ds\\ 
    &\leq  |E|^{-m} + c\, m\int_{r^{\theta}}^{|E|} s^{t+\epsilon-m-1}\, ds\\ 
    & = |E|^{-m} + \frac{c\, m}{m-(t+\epsilon)} \left( r^{\theta(t+\epsilon - m)}-|E|^{\theta(t+\epsilon - m)}\right)\\ 
    & \leq |E|^{-m} + \frac{c\, m}{m-(t+\epsilon)} r^{\theta(t+\epsilon - m)}. 
\end{align*}

Therefore,
\begin{align*}
S_3&\leq |E|^{-m}r^{\theta(m-t) + t} +  \frac{c\, m}{m-(t+\epsilon)} r^{\theta\epsilon}r^{t}\\
&\leq \left( |E|^{-m} +  \frac{c\, m}{m-(t+\epsilon)} \right) r^t.
\end{align*}

Letting $C=3\max\left\{ c+\frac{c\, t}{\epsilon},  |E|^{-m} +  \frac{c\, m}{m-(t+\epsilon)} \right\}$, we obtain
\[
\left(C_{r,\theta}\pot{t,m}(E)\right)\pot{-1}\leq C\, r^t
\quad \text{which implies}
\quad 
\frac{\log C_{r,\theta}\pot{t,m}(E)}{- \log r} - t\geq \frac{\log C}{\log r} \ \text{for all} \ r.
\]
Taking the $\liminf$ of both sides yields
\[
t\leq \underline{\dim}_{\theta}^{m}E.
\]
Letting $t\to \min\left\{ m,d(E)\right\}$ completes the proof.
\end{proof}

\begin{theo}\label{main theo 2}
    Let $E\subset\rd$ be compact, $0<t< d$, and $\theta\in (0,1]$. Then:
    \begin{center}
        If $\underline{\dim}_\theta^t E = t$, then $t\leq \lqh E$,
    \end{center}
    and
    \begin{center}
        If $\overline{\dim}_\theta^t E = t$, then $t\leq \uqh E$.
    \end{center}
\end{theo}

\begin{proof}
    Let $0<\alpha<1$ and $0<m\leq d$ such that $t=\alpha m$.

    Let $B_{\alpha}:\rd\to\mathbb{R}^m$ be the index-$\alpha$ fractional Brownian motion. Then, using Theorem \ref{dimension of Fractional Brownian motion},
    we have $$m=\da B_\alpha(E)\leq \lbd B_\alpha (E)\leq m,$$ almost surely, and hence $\lbd B_\alpha (E) = m$. Now, using Corollary \ref{banajis bound}, we obtain that $\da B_\alpha (E)= \lbd B_\alpha (E)=m$ for all $\theta \in (0,1)$. Using Theorem \ref{dimension of Fractional Brownian motion} again and the fact that $\underline{\dim}_\theta^sE \leq \da E$ for all $0\leq s\leq d$, we have for all $\theta\in(0,1)$,
    $$m=\da B_\alpha (E)\leq \frac{\da E}{\alpha}.$$
    Letting $\theta\to 0$, we have $$t=\alpha m\leq \displaystyle\lim_{\theta\to 0}\da E,$$
    which completes the proof for $\da$. The proof for the $\uda E$ part is analogous.
\end{proof}

We have the following corollary:
\begin{cor}\label{herramienta 2}
    Let $E\subset\rd$ be compact, $\theta\in(0,1]$, and $0<t<d$. Then:
    $$\underline{\dim}_\theta^t E = t\iff t\leq \lqh E$$
    and
    $$\overline{\dim}_\theta^t E = t\iff t\leq \uqh E.$$
\end{cor}
\begin{proof}
    This corollary follows directly from the proof of Theorem \ref{main theorem}, Theorem \ref{main theo 2}, and the fact that $\displaystyle\lim_{\theta\to 0}\overline{\dim}^{t}_{\theta} E \leq \overline{\dim}^{t}_{\theta} E \leq t$ for all $\theta$, with the same holding when replacing $\overline{\dim}^{t}_{\theta}$ with $\underline{\dim}^{t}_{\theta}$.
\end{proof}

Using Theorem \ref{Projection theorem for intermediate dimension} and Theorem \ref{dimension of Fractional Brownian motion}, we obtain the following results about orthogonal projections and fractional Brownian motion.

The following corollary was already proved in the case of compact sets whose $\theta$-intermediate dimension is continuous at $\theta = 0$; see \cite[Corollary 6.4]{Burrell_2021}. Our result improves upon this by extending it to general compact sets.

\begin{cor}
    Let $E\subset\rd$ be compact and let $0<m<d$. Then:
    $$\lbd P_V (E) = m \mbox{ for } \gamma_{d,m}-\mbox{almost all } V\in G(d,m)\iff m\leq \lqh E$$
    and
    $$\ubd P_V (E) = m\mbox{ for } \gamma_{d,m}-\mbox{almost all } V\in G(d,m)\iff m\leq \uqh E.$$
\end{cor}
\begin{proof}
    Combine Corollary \ref{herramienta 2} with Theorem \ref{Projection theorem for intermediate dimension}.
\end{proof}

\begin{cor}
    Let $E\subset\rd$ be compact, $\theta\in(0,1]$, $0<m\leq d$, and $B_\alpha :\rd\to\mathbb{R}^m$ be an index-$\alpha$ fractional Brownian motion $(0<\alpha<1)$. Then:
    $$\da B_\alpha (E) = m \mbox{ almost surely } \iff \alpha m\leq \lqh E$$
    and
    $$\uda B_\alpha (E) = m \mbox{ almost surely }\iff \alpha m\leq \uqh E.$$
\end{cor}
\begin{proof}
    Combine Corollary \ref{herramienta 2} with Theorem \ref{dimension of Fractional Brownian motion}.
\end{proof}

Our next theorem is a lower bound for the $\theta$-intermediate dimension profile in terms of the quasi-Assouad spectrum and the Assouad dimension of the set. This result generalizes the result obtained in \cite{falconer2021assouad} for the upper and lower box dimensions.

\begin{theo}\label{Theo1}
    Let $\theta \in (0,1]$, $\alpha\in(0,1)$, and $E\subset\rn$ be bounded. Then, if $\underline{\dim}_{\theta}\pot{t} E<t$,
    we have
    \be\label{desigualdad assouad y intermedias}\underline{\dim}_{\theta}\pot{t} E\geq \da E-\max\lbrace 0,\dim_A\pot{\alpha}E-t,(\dim_A E-t)(1-\alpha)\rbrace.\ee
    And if $\overline{\dim}_{\theta}\pot{t} E<t$,
    we have
    $$\overline{\dim}_{\theta}\pot{t} E\geq \uda E-\max\lbrace 0,\dim_A\pot{\alpha}E-t,(\dim_A E-t)(1-\alpha)\rbrace.$$
\end{theo}
\begin{proof}
    If $\da E =0$, there is nothing to prove.
    If $\da E= \lbd E\leq t$ for all $\theta$, then
    $$\underline{\dim}_\theta^t E\geq \lim_{\theta\to 0}\underline{\dim}_\theta^t E=\min\{t,\lim_{\theta\to 0 }\da E\}=\lbd E=\da E,$$
    and the result follows.
    If $\da E= \lbd E> t$ for all $\theta$, then
    $$\underline{\dim}_\theta^t E\geq \lim_{\theta\to 0}\underline{\dim}_\theta^t E=\min\{t,\lim_{\theta\to 0 }\da E\}=t,$$
    which contradicts the hypothesis.
    
    So, suppose that $0<\da E <\lbd E$.
    
    Let $A>\dim_A E$ and for $\alpha\in (0,1)$, let $t\neq \dim_A\pot{\alpha}E$ and let $A_\alpha > \dim_A\pot{\alpha}E$ such that $A_\alpha\neq t$.
    
    By definition, there exists a constant $C>0$ such that for all $0<r<R$ and $x\in E$,
    $$N_r(B(x,R)\cap E)\leq C \left( \frac{R}{r}\right)\pot{A},$$
    and for all $0<r\leq R\pot{1/\alpha}<R$ and $x\in E$,
    $$N_r(B(x,R)\cap E)\leq C \left( \frac{R}{r}\right)\pot{A_\alpha}.$$
    
    Let $\theta\in (0,1)$. Let $s<\da E$ and
    $s'<s - \max\left\lbrace 0,A_\alpha - t, (A-t)(1-\alpha)\right\rbrace.$
    
    Note that $s'<\da E$.
    
    For $r\in (0,1)$, let $\mu_r^{s}$ be the Frostman measure of the set $E$ for $s$ of Lemma \ref{Frostman lemma}.
    
    Since $\mu_{r}\pot{s}$ are probability measures supported on $E$, we have
    \begin{align*}
        \left(C_{r,\theta}\pot{s',t}(E)\right)\pot{-1}&\leq \iint \phi_{r,\theta}^{s',t}(x-y)\, d\mu_r^{s}(y) \, d\mu_r^{s}(x)\\
        &=\int \mu_r^{s}(B(x,r))\, d\mu_r^{s}(x)+r^{s'}\iint \frac{\indicator_{(B(x,r^{\theta})\setminus B(x,r))}(y)}{|x-y|^{s'}} \,d \mu_r^{s}(y)\, d\mu_r^{s}(x) \\
        & \qquad
        + r^{\theta(t-s')+s'} \iint \frac{\indicator_{(\rn\setminus B(x,r^{\theta}))}(y)}{|x-y|^{t}}\, d\mu_r^{s}(y) \, d\mu_r^{s}(x)\\
        &=S_1 + S_2 +S_3
    \end{align*}
    respectively.
    
    By the Frostman condition of $\mu_r^{s}$, we have
    $$S_1\leq c\, r^{s}<c\, r^{s'}.$$
    
    For $S_2$, by a change of variable $u=t^{-1/s'}$, we have
    \begin{align*}
        &\int \frac{\indicator_{(B(x,r^{\theta})\setminus B(x,r))}(y)}{|x-y|^{s'}} \, d\mu_r^{s}(y)\\
        &= \int \mu_r^{s} \left\lbrace y:\frac{\indicator_{(B(x,r^{\theta})\setminus B(x,r))}(y)}{|x-y|^{s'}}\geq t\right\rbrace \, dt\\
        &=\int_0^{1/r^{\theta s'}} \mu_r^{s} (B(x,r^{\theta})\setminus B(x,r))\, dt+ \int_{1/r^{\theta s'}}^{1/r^{s'}} \mu_r^{s} \left\lbrace y:\frac{1}{|x-y|^{s'}}\geq t\right\rbrace\, dt \\
        &\leq cr^{\theta (s-s')}+ \int_{1/r^{\theta s'}}^{1/r^{s'}} \mu_r^{s} (B(x,(1/t)^{1/s'}))\, dt\\
        &=cr^{\theta (s-s')}+ s' \int_{r}^{r^{\theta}}u^{-s'-1} \mu_r^{s} (B(x,u)) \, du
        \\
        &\leq cr^{\theta (s-s')}+ \frac{cs'}{s-s'}\left( r^{\theta(s-s') }-r^{s-s'} \right)
        < c+ \frac{cs'}{s-s'},
    \end{align*}
    and then
    $$S_2< \left( c+ \frac{cs'}{s-s'}\right) r^{s'}.$$
    
    For $S_3$, we have
    \begin{align*}
        \int & \frac{\indicator_{(\rn\setminus B(x,r^{\theta}))}(y)}{|x-y|^{t}} \, d\mu_r^{s}(y)
        =\int_{0}^{\infty} \mu_r^{s}\left\lbrace y:\frac{\indicator_{(\rn\setminus B(x,r^{\theta}))}(y)}{|x-y|^{t}}\geq w \right\rbrace\, dw \\
        &=\int_{0}^{r^{-\theta t}} \mu_r^{s}(B(x,w^{-1/t}))\, dw
        =t\int_{r^{\theta}}^{\infty} u^{-t-1}\mu_r^{s}(B(x,u))\, du\\
        &=t\int_{r^{\theta}}^{r^{\alpha \theta}} u^{-t-1}\mu_r^{s}(B(x,u))\, du+t\int_{r^{\alpha\theta}}^{\infty} u^{-t-1}\mu_r^{s}(B(x,u))\, du\\
        &=S_{3,1} + S_{3,2}.
    \end{align*}
    
    For $S_{3,2}$, we have
    \begin{align*}
        &t\int_{r^{\alpha\theta}}^{\infty} u^{-t-1}\mu_r^{s}(B(x,u))\, du\\
        &=t\int_{r^{\alpha\theta}}^{|E|} u^{-t-1}\mu_r^{s}(B(x,u))\, du+t\int_{|E|}^{\infty} u^{-t-1}\mu_r^{s}(B(x,u))\, du\\
        &\leq t\int_{r^{\alpha\theta}}^{\infty} u^{-t-1}\left(\frac{u}{r^{\theta}}\right)^{A_\alpha}\mu_r^{s}(B(x,r^{\theta}))\, du+t\int_{|E|}^{\infty} u^{-t-1}\, du\\
        &\leq t\int_{r^{\alpha\theta}}^{|E|} u^{-t-1}\left(\frac{u}{r^{\theta}}\right)^{A_\alpha}\mu_r^{s}(B(x,r^{\theta}))\, du+|E|^{-t}\\
        &\leq t\int_{r^{\theta}}^{|E|} u^{-t-1}\left(\frac{u}{r^{\theta}}\right)^{A_\alpha}\mu_r^{s}(B(x,r^{\theta}))\, du +|E|^{-t}\\
        &\leq\frac{c\, t}{A_\alpha -t }r^{\theta(s-A_\alpha)}\left( |E|^{A_\alpha -t}-r^{\theta(A_\alpha -t)}\right) +|E|^{-t}\\
        &= \frac{c\, t|E|^{A_\alpha -t}}{A_\alpha -t }r^{\theta(s-A_\alpha)} -\frac{c\, t}{A_\alpha -t } r^{(s -t)} +|E|^{-t}.
    \end{align*}
    
    For $S_{3,1}$, first suppose that $t\neq A$. Then,
    \begin{align*}
        &t\int_{r^{\theta}}^{r^{\alpha \theta}} u^{-t-1}\mu_r^{s}(B(x,u))\, \, du\\
        &\leq t\int_{r^{\theta}}^{r^{\alpha\theta}} u^{-t-1}\left(\frac{u}{r^{\theta}}\right)^{A} \mu_r^{s}(B(x,r^{\theta}))\, \, du\\
        &\leq \frac{c\, t}{A-t}r^{\theta(s-A)}\left(r^{\alpha\theta(A-t)} - r^{\theta(A-t)} \right)\\
        &=\frac{c\, t}{A-t} r^{\theta(s-A)+\alpha\theta(A-t)} - \frac{c\, t}{A-t} r^{\theta(s-t)},
    \end{align*}
    and then
    \begin{equation*}
        \begin{array}{llll}
            S_3 &\leq r^{\theta(t-s')+s'}
            \bigg( \frac{c\, t|E|^{A_\alpha -t}}{A_\alpha -t }r^{\theta(s-A_\alpha)}& + \frac{c\, t}{A-t} r^{\theta(s-A)+\alpha\theta(A-t)}\\
            & & \hspace{-2cm} - \left(\frac{c\, t}{A-t} + \frac{c\, t}{A_\alpha -t }\right) r^{\theta(s-t)} + |E|^{-t}\bigg)\\
            &= \left( \frac{c\, t|E|^{A_\alpha -t}}{A_\alpha -t }\right)r^{\theta(s-(A_\alpha-t)-s')}r\pot{s'} & + \left(\frac{c\, t}{A-t}\right) r^{\theta(s-(1-\alpha)(A-t)-s')}r\pot{s'} \\
            & & \hspace{-2cm} - \left(\frac{c\, t}{A-t} + \frac{c\, t}{A_\alpha -t }\right) r^{\theta(s-s')}r\pot{s'} + |E|^{-t}\, r^{\theta(t-s')+s'}.\\
        \end{array}
    \end{equation*}
    So, by our choice of $s'$, we have
    $$s-(A_\alpha-t)-s'\geq 0\quad \text{ and }\quad s-(1-\alpha)(A-t)-s'\geq 0,$$
    and there exists a constant $C>0$ such that $$S_3\leq C r\pot{s'},$$ and then
    $$\left(C_{r,\theta}^{s',t}(E)\right)\pot{-1}\leq\iint \phi_{r,\theta}^{s',t}(x-y)\, d\mu_r^{s}(x) \, d\mu_r^{s}(y)\leq \left(1+ c+ \frac{cs'}{s-s'} + C\right) r^{s'}.$$
    This implies
    $$\liminf_{r\to 0}\frac{\log C_{r,\theta}^{s',t}(E)}{-\log r}-s'\geq 0,$$
    and finally, using Remark \ref{cota inferior},
    $$\underline{\dim}_{\theta}^{t}E\geq s'.$$
    
    Now, suppose that $A=t$. Then,
    \begin{align*}
        S_{3,1}&=t\int_{r^{\theta}}^{r^{\alpha \theta}} u^{-t-1}\mu_r^{s}(B(x,u))\, \, du
        \leq c\, t\int_{r^{\theta}}^{r^{\alpha \theta}} u^{-t-1} \left(\frac{u}{r\pot{\theta}}\right)^{t} r^{\theta s}\, \, du\\
        &<-c\, t r^{\theta(s-t)} \theta\log(r),
    \end{align*}
    and since $1\leq -\log(r)$ and proceeding as before,
    $$\left(C_{r,\theta}^{s',t}(E)\right)\pot{-1}\leq \iint \phi_{r,\theta}^{s',t}(x-y)\, d\mu_r^{s}(x) \, d\mu_r^{s}(y)\leq -C'\log(r) r^{s'}$$
    for some constant $C'>0$ independent of $r$, which implies
    $$\liminf_{r\to 0}\frac{\log C_{r,\theta}^{s',t}(E)}{-\log r}-s'\geq \liminf_{r\to 0} \frac{\log(\log(r))}{\log(r)}=0.$$
    As before, we conclude that
    $$\underline{\dim}_{\theta}^{t}E\geq s'.$$
    
    Finally, letting $A_\alpha\to \dim_A\pot{\alpha} E$, $A\to \dim_A E$, and $s\to\da E$, we obtain
    $$\underline{\dim}_{\theta}\pot{t} E>s'$$ for all $s'< \da E-\max\lbrace 0,\dim_A\pot{\alpha}E-t,(\dim_A E-t)(1-\alpha)$, which implies,
    \be\label{theo cuando t neq m}
    \underline{\dim}_{\theta}\pot{t} E\geq \da E-\max\lbrace 0,\dim_A\pot{\alpha}E-t,(\dim_A E-t)(1-\alpha)\rbrace.
    \ee
    The case $t=\dim_A\pot{\alpha}E$ follows from \eqref{theo cuando t neq m} and by the continuity in $t$ of $\underline{\dim}_{\theta}\pot{t} E$; see \cite[Corollary 3.8]{Brownian1}.
    The second part of the theorem, concerning the upper $\theta$-intermediate dimension, is similar.
\end{proof}

The next corollaries are immediate consequences of the above theorem. The case for $\theta =1$ was already proved in \cite[Corollary 1.3]{falconer2021assouad}.

\begin{cor}\label{cor for proj}
    Let $E\subset\rd$ be bounded and suppose that $\dim_{qA}E\leq \max\{m,\uda E\}$. Then,
    $$\uda P_V E=\min\{m,\uda E\},$$
    for all $\theta\in (0,1]$ and $\gamma_{n,m}-$almost all $V\in G(n,m)$. More generally,
    \be\label{ineq qA}\uda P_V E \geq \uda E -\max \left\lbrace 0, \dim_{qA}E - m \right\rbrace ,\ee
    for all $\theta\in (0,1]$ and $\gamma_{n,m}-$almost all $V\in G(n,m)$.
    The same conclusion holds with $\uda$ replaced with $\da$.
\end{cor}
\begin{proof}
    Let $\alpha\to 1$ in Theorem \ref{Theo1}.
\end{proof}

From Theorem~\ref{main theorem}, the following corollary is immediate.

\begin{cor}\label{cor1}
    Let \( E \subset \mathbb{R}^n \) be a bounded set such that \( \overline{\dim}_\mathrm{B} E = \uqh E \). Then, for \(\gamma_{n,m}\)-almost every \( V \in G(n,m) \), we have
    \[
    \overline{\dim}_\mathrm{B} (P_V E) = \min\{m, \overline{\dim}_\mathrm{B} E\},
    \]
    and similarly,
    \[
    \underline{\dim}_\mathrm{B} (P_V E) = \min\{m, \underline{\dim}_\mathrm{B} E\}
    \]
    whenever \( \underline{\dim}_\mathrm{B} E = \lqh E \).
\end{cor}

By \cite[Corollary 1.3]{falconer2021assouad}, we have that if $\ubd E=\dim_{qA}E$, then $\ubd P_V E=\min\{\ubd E, m\}$ for a.e. $V\in G(d,m)$, and the same holds for $\lbd$. A natural question is which hypothesis is more general: $\dim_B E=\dim_{qA}E$ or $\dim_B E=\dim_{qH} E$. 

In the following theorem, we will prove that if $\ubd E = \dim_{qA}E$, then $\uda E=\ubd E$, and if $\lbd E=\dim_{qA}E$, then $\da E=\lbd E$ for all $\theta \in (0,1]$. The converse is not generally true; in fact, it is not hard to construct a Cantor set $C\subset \mathbb{R}$ such that $\dim_H C=\lbd C< \dim_{qA }C$. Therefore, Corollary \ref{cor1} is more general than the result obtained in \cite{falconer2021assouad}.

\begin{theo}\label{Angelini ineq}
    Let $E\subset\rd$ be bounded. Then, if $\lbd E = \dim_{qA}E$, we have
    $$\da E = \lbd E\quad \forall\ \theta\in (0,1],$$ 
    and if $\ubd E = \dim_{qA}E$, then 
    $$\uda E = \ubd E\quad \forall\ \theta\in (0,1].$$
\end{theo}

\begin{proof}
    We will prove the first case. The case of $\uda$ is analogous.\\
    If $\da E = 0$ for some $\theta$, then using Lemma \ref{banajis bound}, we have that the intermediate dimension is constantly $0$. So suppose that $0<\da E$.
    By \cite[Corollary 2.8]{banaji2021attainable}, we have that if $\da E=\dim_A E$ for some $\theta$, then the intermediate dimensions are constant and equal to the Assouad dimension. So suppose that $\da E < \dim_A E$ for all $\theta\in(0,1]$.   

    The idea of the proof is to integrate appropriate kernels $\tilde{\phi}_{r,\theta}^{t}$ with respect to a Frostman measure and show that this integral is bounded above.

    Let $0< \theta< 1$ and $r\in (0,1)$. Suppose that $\lbd E=\dim_{qA}E$. Then we have $\lbd E = \dim_{A}^{\alpha}E$ for all $\alpha\in [0,1]$. Using the continuity of the intermediate dimensions, choose $\rho$ and $\alpha$ sufficiently close to $1$ such that $\underline{\dim}_{\rho}E-\dim_A E (1-\alpha)>0$ and $\rho\underline{\dim}_{\rho}E-\dim_{A}^{\alpha}E(\rho-\theta)>0$.\\
    Now let $d_{\rho},d_{\theta}, A_{\alpha}, A_{\rho}, A$ be such that $0<d_\rho < \underline{\dim}_{\rho}E<A<\dim_A E$, $0<d_\theta < \underline{\dim}_{\theta}E$, $0<A_\alpha < \dim^{\alpha}_{A}E$, and $0<A_\rho <\dim_A^{\rho}E$, and satisfies  
    $$0<t:=\min\left\{ \frac{d_\rho - A\, (1-\alpha)}{\alpha},\frac{\rho\, d_{\rho}-A_\alpha\, (\rho-\theta)}{\theta} \right\}.$$
    Note that $t<d_\rho$. By the definition of the Assouad dimension and the Upper Assouad spectrum, there exists $C>0$ such that for all $0<r<R<1$ and $x\in E$, 
    \begin{equation}\label{Assouad} N_r(B(x,R)\cap E)\leq C \left(\frac{R}{r}\right)^A,\end{equation}
    and for all $0<r\leq R^{1/\alpha}<R<1$ and $x\in E$, 
    \begin{equation}\label{Quasi assouad 2}N_r(B(x,R)\cap E)\leq C \left(\frac{R}{r}\right)^{A_\alpha},\end{equation}
    and finally, for all $0<r\leq r^{1/\rho}<R<1$ and $x\in E$, 
    \begin{equation}\label{Quasi assouad}N_r(B(x,R)\cap E)\leq C \left(\frac{R}{r}\right)^{A_{\rho}} .\end{equation}

    Let $\mu_r$ be a Borel probability measure supported on $E$ that satisfies the condition that there exists $C'>0$ such that
    $$\mu_r(B(x,\delta))\leq C\, \delta^{d_{\rho}}$$
    for all $\delta\in [r,r^{\rho}]$ and $x\in \rd$. 
    We will now integrate the kernels $\widetilde{\phi}_{r,\theta}^{t}$ with respect to $\mu_r$ and bound this integral essentially by $r^{t-(\rho-\theta)(A_\rho - t)}$.

    Then,
    \begin{align*}
        &\iint \tilde{\phi}^{t}_{r,\theta}(x-y)\,d\mu_r(x)d\mu_r(y)\\
        &=\iint \mu_r (B(x,r))\,d\mu_r(x) +r^t \, \iint \frac{\indicator_{(B(x,r^\theta)\setminus B(x,r))}}{|x-y|^t}\, d\mu_r(x)d\mu_r(y)\\
        &=S_1 + S_2,
    \end{align*}
    respectively.\\
    By the Frostman condition of $\mu_r$ and since $t<d_{\rho}$, we have
    $$S_1 \leq C'\, r^{d_{\rho}}<C'\, r^{t} \leq C'r^{t - (\rho-\theta)(A_{\rho}-t)}.$$

    For $S_2$, we will use inequality $\eqref{Quasi assouad}$ and the Frostman condition on $\mu_r$ since $t<d_{\rho}$. We have
    \begin{align*}
        \int &\frac{\indicator_{(B(x,r^\theta)\setminus B(x,r))}}{|x-y|^t}\,d\mu_r (y)\\
        &= \int_{0}^{\infty} \mu_r \left(\left\{y:\frac{\indicator_{(B(x,r^\theta)\setminus B(x,r))}}{|x-y|^t}\geq u \right\}\right) \,du\\
        &= \int_{0}^{1/r^{\theta t}} \mu_r \left(B(x,r^\theta)\setminus B(x,r)\right) \,du+\int_{1/r^{\theta t}}^{1/r^t} \mu_r \left(B(x,u^{-1/t})\right) \,du\\
        &\leq \int_{0}^{1/r^{\theta t}} \mu_r \left(B(x,r^\theta)\right) \,du+t \int_{r}^{r^{\theta}} u^{-1-t}\mu_r \left(B(x,u)\right) \,du\\
        &\leq C \left(\frac{r^\theta}{r^{\rho}}\right)^{A_{\rho}}\int_{0}^{1/r^{\theta t}} \mu_r \left(B(x,r^\rho)\right) \,du+t \int_{r}^{r^{\theta}} u^{-1-t}\mu_r \left(B(x,u)\right) \,du\\
        &\leq C'C \left(\frac{r^\theta}{r^{\rho}}\right)^{A_{\rho}}\int_{0}^{1/r^{\theta t}} r^{\rho d_{\rho}} \,du+t \int_{r}^{r^{\theta}} u^{-1-t}\mu_r \left(B(x,u)\right) \,du\\
        &\leq C'C r^{(\rho-\theta)(t-A_{\rho})} +t \int_{r}^{r^{\theta}} u^{-1-t}\mu_r \left(B(x,u)\right) \,du.
    \end{align*}
    To bound the last term of the sum, we will partition the interval of integration $I = (r,r^{\theta})$ into three parts. One interval in which we can apply the Frostman condition of $\mu_r$, and for the remaining two intervals, we will utilize the properties of the Quasi-Assouad spectrum and Assouad dimension. Namely, 
    $$I = (r,r^{\theta})=(r,r^{\rho})\cup [r^{\rho},r^{\alpha\rho})\cup [r^{\alpha\rho},r^{\theta}).$$
    Hence, 
    $$t \int_{r}^{r^{\theta}} u^{-1-t}\mu_r \left(B(x,u)\right) \,du = S_{2,1}+S_{2,2}+S_{2,3}.$$

    Using the Frostman condition, it is immediate that $$S_{2,1}\leq \frac{t}{d_{\rho}-t}r^{\rho(d_{\rho}-t)}\leq \frac{t}{d_{\rho}-t}.$$ For $S_{2,2}$, we use inequality \eqref{Assouad}, the Frostman condition for $\mu_r$, and the fact that $t\leq \frac{d_{\rho} -A(1-\alpha)}{\alpha}$. We have
    \begin{align*}
        t \int_{r^{\rho}}^{r^{\alpha\rho}} u^{-1-t}\mu_r \left(B(x,u)\right) \,du &\leq C t \int_{r^{\rho}}^{r^{\alpha\rho}} u^{-1-t} \left( \frac{u}{r^{\rho}}\right)^A \mu_r \left(B(x,r^{\rho})\right) \,du\\
        &\leq \frac{C C't}{A-t} r^{-\rho (A- d_{\rho})} r^{(\alpha\rho)(A-t)}\leq\frac{C C't}{A-t}.
    \end{align*}
    Finally, for $S_{2,3}$, we need to use inequality \eqref{Quasi assouad 2} and the fact that $t\leq \frac{\rho\, d_{\rho} - A_{\alpha}(\rho-\theta)}{\theta}$ to obtain
    $$t \int_{r^{\alpha\rho}}^{r^{\theta}} u^{-1-t}\mu_r \left(B(x,u)\right) \,du\leq \frac{C C't}{A-t}.$$
    Choosing $M=3\max\{\frac{C C't}{A-t}, \frac{t}{d_{\rho}-t}, CC'\}$, we have
    $$S_2\leq r^t\, M\, \left(r^{(\rho-\theta)(t-A_{\rho})} + 1\right)\leq 2M r^{t - (\rho-\theta)(A_{\rho}-t)}.$$

    Therefore, 
    $$\iint \tilde{\phi}^{t}_{r,\theta}(x-y)\,d\mu_r(x)d\mu_r(y)\leq 2M r^{t - (\rho-\theta)(A_{\rho}-t)}.$$

    By Lemma \ref{lemma intermediate dimension}, this implies
    $r^{(\rho-\theta)(A_{\rho}-t)} \leq S_{r,\theta}^{t}(E)$. Taking logarithms and letting $r\to 0$, we have
    $$(\rho-\theta)(t - A_{\rho})\leq \liminf_{r\to 0}\frac{\log S_{r,\theta}^{t}(E)}{-\log r}.$$
    Now, since by \cite[Lemma 2.1]{Burrell_2021} the right-hand side of the inequality is continuous in $t$, we can let $\alpha\to 1$, $d_{\rho} \to \underline{\dim}_{\rho} E$, and $A_{\rho}\to\dim_{A}^{\rho}E$, obtaining $t=\dim_{qA}E - \frac{\dim_{qA} E- \underline{\dim}_{\rho} E}{\theta}$. Finally, letting $\rho \to 1$ and using the assumption that $\lbd E=\dim_{qA} E$, we have 
    $$0\leq \displaystyle\liminf_{r\to 0}\frac{\log S_{r,\theta}^{\lbd E}(E)}{-\log r},$$ which implies $\lbd E\leq \da E$, and the result follows.
\end{proof}

Our last theorem is another lower bound for intermediate dimension profiles. This lower bound generalizes the one obtained in \cite[Proposition 2.12]{falconer2018capacity} for upper and box dimension profiles, although the result in that reference is stated in a different form. After rearranging the inequality presented there, one can recover a similar bound to the one we obtain here for intermediate $\theta$-dimensions. The approach to the proof is also closely related.

\begin{theo}\label{Theo2}
    Let $E\subset\rd$ be bounded, $\theta\in(0,1)$, and $t\leq s$. Then, 
    $$\underline{\dim}_{\theta}\pot{t}E\geq \frac{\underline{\dim}_{\theta}\pot{s}E}{1 + \left(\frac{1}{t} - \frac{1}{s}\right)\underline{\dim}_{\theta}\pot{s}E}$$
    and
    $$\overline{\dim}_{\theta}\pot{t}E\geq \frac{\overline{\dim}_{\theta}\pot{s}E}{1 + \left(\frac{1}{t} - \frac{1}{s}\right)\overline{\dim}_{\theta}\pot{s}E}.$$
\end{theo}
\begin{proof}
    This proof is inspired by the argument used in \cite[Proposition 2.12]{falconer2018capacity}, although the details differ due to the generality of the intermediate dimension setting.

    Let $d_t < \underline{\dim}_{\theta}^{t} E$. Then, we have 
    \begin{equation}\label{ineq theo 2}
        r^{-d_t} \leq C_{r,\theta}^{d_t,t}
    \end{equation}
    for all sufficiently small $r$.\\

    Let $d_s = \dfrac{d_t}{1 + \left(\frac{1}{s} - \frac{1}{t}\right) d_t}$, $r > 0$, and let $R \geq r$ be such that $R^{d_t / d_s} = r$, i.e., 
    \[
    R = r^{\frac{\theta s + d_s (1-\theta)}{\theta s + d_t \left(1 - \theta \frac{s}{t}\right)}}.
    \] 

    Let $\mu$ be a Borel probability measure supported on $E$ such that 
    \begin{equation}\label{phi=C}
        \int \phi_{r,\theta}^{d_t, t}(x - y)\, d\mu(y) = \left(C_{r,\theta}^{d_t, t}(E)\right)^{-1}
    \end{equation}
    for $\mu$-almost all $x$. Now, integrating, we get:
    \[
    \begin{split}
        \int \phi_{r,\theta}^{d_s,s}(x - y)\, d\mu(y) 
        &\leq \mu(B(x,R)) \\
        &\quad + \int_{R \leq |x - y| \leq R^\theta} \frac{r^{d_s}}{|x - y|^{d_s}}\, d\mu(y) \\
        &\quad + \int_{|x - y| > R^\theta} \frac{r^{\theta(s - d_s) + d_s}}{|x - y|^s}\, d\mu(y) \\
        &= S_1 + S_2 + S_3.
    \end{split}
    \]

    For $S_1$, using the definition of $R$, we obtain
    \begin{equation}\label{S1}
    \begin{split}
        \mu(B(x,R)) 
        &\leq R^{d_t} \left(R^{-d_t} \int \phi_{R,\theta}^{d_t, t}(x - y)\, d\mu(y)\right) \\
        &= r^{d_s} \left(R^{-d_t} \int \phi_{R,\theta}^{d_t, t}(x - y)\, d\mu(y)\right).
    \end{split}
    \end{equation}

    For $S_2$, by Hölder's inequality:
    \begin{equation}\label{S2}
    \begin{split}
        \int_{R \leq |x - y| \leq R^\theta} \frac{r^{d_s}}{|x - y|^{d_s}}\, d\mu(y)
        &\leq r^{d_s} \left(\int \frac{1}{|x - y|^{d_t}}\, d\mu(y)\right)^{d_s/d_t} \\
        &\leq r^{d_s} \left(R^{-d_t} \int \phi_{R,\theta}^{d_t, t}(x - y)\, d\mu(y)\right)^{d_s/d_t}.
    \end{split}
    \end{equation}

    Finally, for $S_3$, again using Hölder:
    \begin{equation}\label{S_3}
    \begin{split}
        \int_{|x - y| > R^\theta} \frac{r^{\theta(s - d_s) + d_s}}{|x - y|^s}\, d\mu(y) 
        &\leq r^{\theta(s - d_s) + d_s} R^{-\theta s} \left(\int \frac{R^{\theta t}}{|x - y|^t}\, d\mu(y)\right)^{s/t} \\
        &= r^{d_s} \left(R^{-d_t} \int \phi_{R,\theta}^{d_t, t}(x - y)\, d\mu(y)\right)^{s/t},
    \end{split}
    \end{equation}
    by our choice of $d_s$.

    Now, using \eqref{ineq theo 2}, \eqref{phi=C}, and combining \eqref{S1}, \eqref{S2}, and \eqref{S_3}, we get:
    \[
        \left(C_{r,\theta}^{d_s, s}(E)\right)^{-1} \leq \int \phi_{r,\theta}^{d_s,s}(x - y)\, d\mu(y) \leq 3r^{d_s}
    \]
    for all sufficiently small $r$. Hence,
    \[
        \liminf_{r \to 0} \frac{\log C_{r,\theta}^{d_s, s}(E)}{-\log r} - d_s \geq 0,
    \]
    which implies
    \[
        \underline{\dim}_{\theta}^{s} E \geq d_s = \frac{d_t}{1 + \left(\frac{1}{s} - \frac{1}{t}\right) d_t}.
    \]

    The result follows by letting $d_t \to \underline{\dim}_{\theta}^{t} E$. The case of the upper intermediate dimension profile is similar.
\end{proof}

\begin{rmk}
    The lower bounds in the previous theorem are attained for some sets \( E \), as shown in the next example.

    Let $0<s\leq t\leq 1$, $F_p=\left\lbrace \frac{1}{n\pot{p}} :n\in\mathbb{N} \right\rbrace$, and $\theta\in (0,1)$.
    By Theorem 1.1 in \cite{falconer2021intermediate} and Theorem 3.4 in \cite{Burrell_2021}, we have $\underline\dim_\theta\pot{s} F_p = \overline\dim_\theta\pot{s} F_p=\frac{s\theta}{\theta + sp}$ for any $s\in (0,1]$. Therefore, for $s\geq t$, we have
    \begin{align*}
        \frac{\underline{\dim}_{\theta}\pot{s}F_p}{1 + \left(\frac{1}{t} - \frac{1}{s}\right)\underline{\dim}_{\theta}\pot{s}F_p} &= \frac{s\theta}{\theta + sp}\left(1 + \left(\frac{1}{t} - \frac{1}{s}\right)\frac{s\theta}{\theta + sp}\right)^{-1}\\
        &= \left(\frac{ts\theta}{\theta + sp}\right)\left(\frac{t(\theta+sp)+s\theta-t\theta}{\theta+sp}\right)^{-1}\\
        &= \frac{t\theta}{\theta+tp} = \underline{\dim}_{\theta}^{t} F_P.
    \end{align*}
    Thus, the lower bound given in Theorem \ref{Theo2} is attained for this set.
\end{rmk}

Moreover, combining Theorem \ref{Theo2} and Corollary \ref{cor for proj}, we have for $t\leq\dim_{qA}E$,
\begin{equation*} \underline{\dim}_\theta^{t}E\geq\frac{\da E}{1+\left(\frac{1}{t} - \frac{1}{\dim_{qA}E}\right)\da E}\end{equation*}
and 
\begin{equation}\label{last ineq} \overline{\dim}_\theta^{t}E\geq\frac{\uda E}{1+\left(\frac{1}{t} - \frac{1}{\dim_{qA}E}\right)\uda E}.\end{equation}

One may wonder when \eqref{last ineq} improves inequality \eqref{ineq qA} for sets $E$ such that $0<\uda E<\dim_{qA}E$ and $0<t< \dim_{qA} E$, i.e., when 
\begin{equation}\label{question}\uda E-\left(\dim_{qA}E-t\right)< \frac{\uda E}{1+\left(\frac{1}{t} - \frac{1}{\dim_{qA}E}\right)\uda E}\leq\overline{\dim}_\theta^{t}E\end{equation}
holds. 

A direct algebraic manipulation of the last inequality implies that inequality \eqref{question} holds if and only if
\begin{align*}
    t^{2}\left(\dim_{qA}E-\uda E\right) &+ t\left(2\dim_{qA}E\uda E - (\uda E)^{2}-(\dim_{qA}E)^{2}\right)\\
    &+ \left(\dim_{qA}E(\uda E)^{2}-(\dim_{qA}E)^{2}\uda E\right) < 0.
\end{align*}

But this is true if and only if
$$t^{2}- t\left(\dim_{qA}E-\uda E\right)-\dim_{qA}E\uda E < 0.$$
This is a second-degree polynomial in $t$ with zeros at $t=-\uda E$ and $t=\dim_{qA}E$. Therefore, for all $t\in(0,\dim_{qA}E)$, inequality \eqref{question} holds for all sets with $0<\uda E<\dim_{qA}E$. 

We finish this work with a simple corollary that improves Corollary 3.8 in \cite{Brownian1}.
\begin{cor}\label{lipschitz cor}
    Let $E\subset\rd$ be bounded and $\theta\in (0,1]$. The functions $f,g:(0,d)\to [0,d]$ defined by $$f(t)=\underline{\dim}_\theta\pot{t}E$$ and $$g(t)=\overline{\dim}_\theta\pot{t}E$$ are Lipschitz functions.
\end{cor}
\begin{proof}
    For all $0<t\leq d$, we have 
    \begin{equation}\label{eq for lipschitz}\underline{\dim}_{\theta}\pot{t}E\leq t.\end{equation}
    Now, using Theorem \ref{Theo2} and \eqref{eq for lipschitz}, we have for $0<s\leq t$,
    $$\underline{\dim}_{\theta}\pot{t}E - \underline{\dim}_{\theta}\pot{s}E \leq \left(\frac{\underline{\dim}_{\theta}\pot{s}E \underline{\dim}_{\theta}\pot{t}E}{st}\right)(t-s)\leq t-s.$$
\end{proof}

\section*{Final remarks and broader applicability}

\subsection*{Broader applicability}
We chose to present our results for the case of orthogonal projections and fractional Brownian motion. However, in \cite{Brownian1}, it is shown that  the intermediate dimension profiles also  provide meaningful bounds for intermediate dimensions for a more general set of functions:
\begin{defi}
Let $(\Omega,\mathcal{F},P)$ be a probability space, and let $(\omega,x)\mapsto f_{\omega}(x)$ be a $\sigma(\mathcal{F}\times\mathcal{B})$-measurable function, where $\mathcal{B}$ denotes the Borel $\sigma$-algebra on $\mathbb{R}^d$. Let $\mathcal{F}_1 = \{f_\omega : E \to \mathbb{R}^m, \omega \in \Omega\}$ be a family of continuous functions, measurable with respect to $\sigma(\{F \times B : F \in \mathcal{F}, B \in \mathcal{B}\})$, that satisfies one of the following conditions:
\begin{enumerate}
    \item \label{f-lower} There exists a constant $c > 0$ such that for all $x, y \in E$ and $r > 0$,
    $$
    P\left(\left\{\omega : |f_\omega(x) - f_\omega(y)| \leq r \right\}\right) \leq c\, \phi_{r^{\gamma},\theta}^{m/\gamma, m/\gamma}(x - y),
    $$
    \item \label{label-upper} For all $\omega \in \Omega$, there exists a constant $c_\omega > 0$ such that for all $x, y \in E$,
    $$
    |f_\omega(x) - f_\omega(y)| \leq c_\omega |x - y|^{1/\gamma}.
    $$
\end{enumerate}
\end{defi}
Therefore all our results that involve intermediate dimension profiles can be directly applied to this more general setting, since in
\cite[Theorems 3.1 and 3.3]{Brownian1}, it is shown exactly how the intermediate dimension profiles control the dimension of images under functions in the family $\mathcal{F}_1$. 

Precisely: {\em 
Let $E \subset \mathbb{R}^d$ be compact, $\theta \in (0,1]$, $\gamma \geq 1$, and $m \in \mathbb{N}$. If $\mathcal{F}_1$ satisfies condition \eqref{f-lower}, then for $P$-almost all $\omega \in \Omega$, $f_\omega \in \mathcal{F}_1$ we have
\[
\underline{\dim}_\theta f_\omega(E) \geq \gamma\, \underline{\dim}_{\theta}^{m/\gamma} E \quad \text{and} \quad \overline{\dim}_\theta f_\omega(E) \geq \gamma\, \overline{\dim}_{\theta}^{m/\gamma} E.
\]
If $\mathcal{F}_1$ satisfies condition \eqref{label-upper}, then for all $\omega \in \Omega$, $f_\omega \in \mathcal{F}_1$ we have
\[
\underline{\dim}_\theta f_\omega(E) \leq \gamma\, \underline{\dim}_{\theta}^{m/\gamma} E \quad \text{and} \quad \overline{\dim}_\theta f_\omega(E) \leq \gamma\, \overline{\dim}_{\theta}^{m/\gamma} E.
\]}

\subsection*{Size of exceptional directions}

We now include two corollaries that quantify the \emph{size} of the set of exceptional directions where the projected dimension drops. These are direct consequences of our main theorem combined with Theorem \ref{exceptional directions}, and although the proofs are straightforward, we believe they are of independent interest and help to clarify the scope of the results.

\begin{cor}\label{exceptional 1}
    Let $E\subset\rd$ be compact and $m\geq\dim_{qA} E$. Then,
    $$
    \dimh\left\{V\in G(d,m):\uda P_V E<\uda E\right\}\leq m(d-m)-(m-\dim_{qA}E)
    $$
    and
    $$
    \dimh\left\{V\in G(d,m):\da P_V E<\da E\right\}\leq m(d-m)-(m-\dim_{qA}E).
    $$
\end{cor}

\begin{proof}
    Apply Corollary \ref{cor for proj} in Theorem \ref{exceptional directions}.
\end{proof}

\begin{cor}\label{excepcional 2}
    Let $E\subset\rd$ be bounded, $1\leq m\leq d$, and $\lambda\in (0,\uqh E)$. Then,
    $$
    \dimh\left\{V\in G(d,m):\uda P_V E<\lambda\right\} \leq m(d-m)-(m-\lambda),
    $$
    and if $\lambda\in (0,\lqh E)$, then 
    $$
    \dimh\left\{V\in G(d,m):\da P_V E<\lambda\right\} \leq m(d-m)-(m-\lambda).
    $$
\end{cor}

\begin{proof}
    Apply Corollary \ref{herramienta 2} in Theorem \ref{exceptional directions}.
\end{proof}


\subsection*{Acknowledgments} 
We thank the anonymous referee for the careful reading of the manuscript and in particular for suggesting the introduction of the term {\em quasi Hausdorff} dimension.

The research of the authors is partially supported by Grants
PICT 2022-4875 (ANPCyT), 
PIP 202287/22 (CONICET), and 
UBACyT 2022-154 (UBA). Nicolas Angelini is also partially supported by PROICO 3-0720 ``An\'alisis Real y Funcional. Ec. Diferenciales''.
\bibliographystyle{plain}
\bibliography{Critical_values_for_Intermediate_and_Box_dimension_of_projections_and_other_images}
\end{document}